\documentclass[11pt]{article}

\newtheorem{theorem}{Theorem}

\newtheorem{lemma}[theorem]{Lemma}

\newcommand{\R}{\mathbb{R}}
\newcommand{\M}{\mathcal{M}}

\newcommand{\I}{\mathbb{I}}
\newcommand{\Pp}{\mathbb{P}}

\newcommand{\E}{{\mathbb{E}}}
\newcommand{\N}{{\mathbb{N}}}

\newcommand{\K}{\mathbf{K}}

\newcommand{\pf}{\operatorname{pf}}
\newcommand{\qed}{\vrule height7.5pt width4.17pt depth0pt}

\usepackage{amsmath} 
\usepackage{amssymb}
\usepackage{hyperref} 
\usepackage{graphicx}
\begin{document}

\title{Entrance laws for coalescing and annihilating Brownian motions.}

\date{}

\author{Roger Tribe and Oleg Zaboronski}

\maketitle
\begin{abstract}
Systems of instantaneously annihilating or coalescing Brownian motions on the line are considered. The extreme points of 
the set of entrance laws for this process are shown to be Pfaffian point processes at all times and their kernels are identified.

\end{abstract}


We consider a system of particles moving, between collisions, as independent Brownian motions on  $\R$. 
A pair of particles upon collision instantaneously react, either annihilating each other 
with probability $\theta$ or coalescing into a single particle with probability $1-\theta$. 
Separate collisions produce independent reactions. Thus the parameter 
$\theta \in [0,1]$ interpolates between the purely coalescing case and the purely 
annihilating case.

We will write $X_t$ for the empirical measure of the positions of particles at time $t$. 
A suitable state space for this system $(X_t)$ is $\M$, the space of locally finite
point measures on $\R$ with the topology of vague convergence. The instantaneous reactions mean that
$X_t$ takes values in the subset $\M_0 \subseteq \M$ of simple point processes, that is there is at 
most one particle at any point. The process has a Feller transition density 
$(p_t(\mu, d \nu):t \geq 0)$ and the process has a Markov family of laws 
$(\E_Q: Q \; \mbox{a probability supported on $\M_0$})$. For the coalescing case the 
process can be constructed from finite systems using monotonicity, that is by adding one more initial particle at a time, while the general case requires more care. The details of all these basic properties are, for $\theta = 0,1$, in the Appendix 4.1 to \cite{TZ}. 
For the mixed case $\theta \in (0,1)$, which arises in several settings (polymer chains, multi-valued voter models, coalescent models - see \cite{GPTZ}), these basic results above still hold by 
repeating the arguments for $\theta =1$ from  \cite{TZ} replacing the duality function used there with the 
general $\theta$ duality function (\ref{thetaduality}) below. 

The system is often studied starting from a specific entrance law, informally starting with a particle 
at every $x \in \R$ (for example this is the case with the Arratia flow). The intuition is that the instantaneous reactions bring the system instantaneously into $\M_0$. However a complete 
description of entrance laws requires care. See \cite{hammer2021entrance} where a complete
classification of the entrance laws for $\theta=1$ is found using duality with the continuum voter model. However,
the Pfaffian formulae for the intensities of the one dimensional marginals clarify the picture and 
lead to a simple description of  the complete set of entrance laws, which is the purpose of this note.
The usefulness of this Pfaffian property is illustrated in \cite{fitzgerald2022asymptotic} where Fredholm Pfaffians are used to study empty intervals and exit measures. 

We recall the Pfaffian structure of $X_t$, yielding explicit Lebesgue intensities 
$\rho_t(x_1,\ldots,x_n)$ for $X_t$. These were studied first for $\theta =0$ for the Arratia flow,
and the analogous flow for $\theta=1$, in \cite{TZ}, and then more generally in \cite{GPTZ} and \cite{GTZ} for mixed systems, together with associated systems where certain branching or immigration mechanisms are present. Started from a deterministic initial condition
$\mu \in \M_0$ the point process $X_t$ is, at any fixed time $t>0$, a Pfaffian point process. This means that its intensities are given in terms of a Pfaffian
\begin{equation} \label{rhon}
\rho_t(x_1,\ldots,x_n) = \pf(\K_t(x_i,x_j)) \quad \mbox{for $x_1<x_2< \ldots <x_n$ }
\end{equation}
where the kernel $\K_t: \overline{V}_2 \to M_{2 \times 2}(\R) $, where $V_2 = \{(x,y): x < y\}$,
is constructed from the initial condition $\mu$ as we now recall. Firstly $\K_t$ is
in 'derived form', that it is derived from a scalar kernel $K_t: \overline{V}_2 \to \R$ via the 
relation
\begin{equation} \label{derivedform}
\K_t(x,y) = \frac{1}{1+\theta}
\begin{pmatrix}
K_t(x,y) & -D_x K_t(x,y) \\
-D_yK_t(x,y) & D_{xy}K_t(x,y)
\end{pmatrix}
\quad \mbox{for $t>0$ and $x < y$,}
\end{equation}
and $\K_t(x,x)$ is skew-symmetric with $\K_t(x,x)_{1,2}=-\frac{1}{1+\theta}D_xK_t(x,x)$.
The scalar kernel $K \in C^{1,2}((0,\infty) \times \overline{V}_2)$ 
is the unique bounded solution to the heat equation
\begin{equation}  \label{Kpde} 
\left\{ \begin{array}{rcll}
\partial_t K & = & \Delta K & \mbox{on $(0,\infty) \times V_2$,} \\
K_t (x,x)  & = & 1 &  \mbox{for $x \in \R$,}
\end{array} \right.
\end{equation}
satisfying the initial condition $K_t \to K_0$ in distribution as $t \downarrow 0$ on $V_2$, where 
\begin{equation} \label{uIC}
K_0(x,y)  =   (-\theta)^{X_0(x,y)} \quad \mbox{for $(x,y) \in V_2.$} 
\end{equation}
The proof of this Pfaffian structure is based on the Markov duality formula 
\begin{equation} \label{dualityintro}
\E[ (-\theta)^{X_t(x_1,x_2) + X_t(x_3,x_4) + \ldots +X_t(x_{2n-1},x_{2n}) } ]
= \pf( K_t(x_i,x_j): i,j \leq 2n).
\end{equation}
For coalescing-annihilating random walks this is Lemma $7$ of \cite{GPTZ}.
Expression (\ref{dualityintro}) is the corresponding continuous limit obtained
by following the arguments of Section $3$ of that paper.
Here, and throughout, we write $\mu(a,b)$ (and $\mu(a,b]$ e.t.c.) as shorthand for  $\mu((a,b))$. We also
use the convention that $0^k = \I(k=0)$, so that for instance when $\theta =0$ the expression $(-\theta)^{X_0(x,y)}$ becomes the indicator $I(X_0(x,y)=0)$ that there are no particles inside $(x,y)$. 

Write $(T_t)$ for the Markov semigroup of $(X_t)$ acting on bounded measurable $F:\M_0 \to \R$. 
Recall that an entrance law for $(X_t)$ is a family of laws $(Q_t:t>0) $ on $\M_0$ so that
\begin{equation} \label{elaw}
\int_{\M_0} T_t F(\mu) \, Q_{s}(d \mu) = \int_{\M_0} F(\mu) \, Q_{t+s}(d \mu) \quad \mbox{for all $s,t>0$ and all $F$.} 
\end{equation}
Clearly, entrance laws form a convex set and our aim is classify the extreme points of this set. 
\vspace{.1in}


\noindent
\textbf{Notation.} We write $(K^f_t)$ for the solution to the heat equation (\ref{Kpde}) 
with initial condition $K^f_0=f$.  We write $Q^f = (Q^f_t: t >0)$ for the family of laws on $\M_0$, when it exists,
where $Q^f_t$ is the law of the point process with intensities $\rho^{(n)}_t$ given via (\ref{rhon}) for the kernels 
$\K^f_t$ arising as in (\ref{derivedform}) from the scalar kernel $K^f_t$.

As explained above, starting from a deterministic condition, the law of $(X_t:t >0)$ is given by $Q^f$ with
$ f(x,y) = (-\theta)^{X_0(x,y)}$. The aim
of this note is to show that all entrance laws are mixtures of $Q^f$ for suitable functions $f: V_2 \to \R$.  
%

\begin{theorem} \label{t1}
The extreme elements of the set of entrance laws for $(X_t)$ are $(Q^f: f \in C_{\theta})$ where
$C_{\theta} \subseteq L^{\infty}(V_2)$ is given by 
\[
C_1 =  \left\{ f(x) f(y), (x,y)\in V_2: \mbox{measurable} \; f:\R \to [-1,1] \right\},
\]
and for $\theta \in [0,1)$
\[
C_{\theta} = \left\{ \I(S \cap (x,y) = \emptyset), (x,y)\in V_2: \mbox{closed $S \subseteq \R$} \right\}.
\]
Moreover these laws $Q^f$ are distinct, that is $f \neq g$ implies $Q^f \neq Q^g$.   
\end{theorem}  
The duality function used to analyse mixed systems is, for $\theta \in [0,1]$ and $\mu \in \M_0$,  
\begin{equation} \label{thetaduality}
 s_{\mu}(x,y) :=  (-\theta)^{\mu(x,y)} \quad \mbox{for $(x,y) \in V_2.$} 
\end{equation}
Here is the key underlying lemma, whose proof is delayed until the end of this note. 
\begin{lemma}  \label{closure}
The weak-$*$ closure in $L^{\infty}(V_2)$, as dual to
$L^1(V_2)$, of the set
\[
S_\theta= \left\{ s_{\mu}(x,y), (x,y)\in V_2: \mu = \sum_{k=1}^n \delta_{x_k} \in \M_0, n \geq 0\right\}
\]
of finite spin functions
is $\tilde{C}_{\theta}$, where 
$\tilde{C}_{1}= C_1$ and for $\theta \in [0,1)$ it is the set
\[
\tilde{C}_{\theta}  =  \left\{  \I (S_c \cap (x,y)  = \emptyset) (-\theta)^{\sum_{z \in S_i \cap(x,y)} w(z)}: 
\mbox{closed $S \subseteq \R$},  \; w:S_i \to \N \cup \{\infty\} \right\}
\]
where $S = S_i \cup S_c$ is the disjoint decomposition of a closed set $S$ into its isolated points $S_i$ and its 
cluster points $S_c$. 
\end{lemma}

\textbf{Remarks.} \textbf{1.} The superset $\tilde{C}_{\theta} \supseteq C_{\theta}$, when $\theta \in [0,1)$, will label
entrance laws (via the map $f \to Q^f$); however only the set $C_{\theta}$ will label extremal entrance laws. 

\noindent
\textbf{2.} All the functions here lie in the unit ball $B_1 = \{f:V_2 \to [-1,1]\}$ and the 
weak-$*$ topology is metrizable on this ball (since $L^1(V_2)$ is separable). 

\textbf{Proof of Theorem.} 
We first check that $Q^f$, for $f \in \tilde{C}_{\theta}$, do form entrance laws. 
We follow the steps from \cite{TZ} where an entrance law for the cases $\theta \in\{0,1\}$ when $f \equiv 0$ was constructed (called there a 'maximal' entrance law and informally corresponding to starting a particle at every point in $\R$ as for the Arratia flow). 
Fix $f \in \tilde{C}_{\theta}$. By Lemma \ref{closure} we can choose a sequence
$(\mu_n)$ 
so that $s_{\mu_n} \to f$ weak-$*$. 
Let $(X_t^{(n)})$ be the corresponding particle system with
 the initial condition $\mu_n$. At a fixed $t>0$, the corresponding scalar kernels $K^{(n)}_t$, solving 
 (\ref{Kpde}) with initial condition $K^{(n)}_0 = s_{\mu_n}$, are given, for $(x,y) \in \overline{V}_2$, by
\begin{equation} \label{scalarkernel}
K^{(n)}_t(x,y) = 1 + \int_{V_2} (g_t(x-x',y-y') - g_t(y-x',x-y')) (s_{\mu_n}(x',y')-1) dx' dy'
\end{equation}
where $g_t(x,y) = (1/4\pi t) \exp(-(x^2+y^2)/4t)$, $x,y\in \R^2$. 
Using this one sees that $K^{(n)}_t$,
 together with their derivatives
 $D_x K^{(n)}_t, D_y K^{(n)}_t, D_{xy} K^{(n)}_t$, converge bounded pointwise to $K^f_t$ and its associated derivatives. This convergence of the kernels $\K^{(n)}_t$ implies that the associated Pfaffian point processes $X_t^{(n)}$ converge in law to a limiting point process $X_t$ with law $Q_t^f$ 
(see Lemma 10 in \cite{GPTZ}).  Moreover the Markov duality formula extends to hold for the limit, that is when $X_t$ has law $Q^f_t$ 
then 
\begin{equation} \label{duality}
\E[ (-\theta)^{X_t(x_1,x_2) + X_t(x_3,x_4) + \ldots +X_t(x_{2n-1},x_{2n}) } ]
= \pf( K^f_t(x_i,x_j): i,j \leq 2n)
\end{equation}
when $t>0, x_1<x_2 < \ldots <x_{2n}$. Note that this formula implies the distinctness claimed in the theorem; if $f \neq g$ then (at least for small
$t>0$) the kernels $K^f_t$ and $K^g_t$ will be distinct and the duality formula shows that the laws $Q^f$ and $Q^g$ are not equal.  
 
 Passing to the limit $n \to \infty$ in the semigroup property for $(X_t^{(n)})$ for bounded continuous $F$
 \[
 \int_{\M_0} T_t F(\mu) \, \Pp[X^{(n)}_s \in d \mu]  = \int_{M_0} F(\mu) \, \Pp[X^{(n)}_{t+s} \in d \mu],
 \]
and using the Feller property to see that $T_t F$ is still continuous,
we see that $(Q^f_t:t >0)$ satisfies the entrance law equation (\ref{elaw}), finishing the proof that it is an entrance law. 

 We now fix any entrance law $(Q_t: t>0)$ and show that it is a mixture of the 
 entrance laws $(Q^f_t: t>0)$ constructed above. 
The entrance law equation (\ref{elaw}) gives for any $0<r<t$
\begin{equation} \label{extreme10}
\int_{\M_0} F(\mu) Q_t(d \mu) = \int_{\M_0} T_{t-r} F(\mu) Q_r(d \mu) = 
\int_{B_1}  \int_{\M_0} F(\nu) Q^f_{t-r}(d \nu) Q_r(\{\mu: s_\mu \in df\})
\end{equation}
where we have used the fact that started from $\mu$ the law of $X_{t-r}$ is $Q^f_{t-r}$ for 
$f = s_{\mu}$. 
The pushforward law $Q_r(\{\mu: s_\mu \in df\})$ on $L^{\infty}(V_2)$ is supported on 
the unit ball $B_1 \subseteq L^{\infty}(V_2)$ since 
$|s_{\mu}| \leq 1$. The unit ball $B_1$ is weak-$*$ compact by Alaoglu's Theorem (and metrizable as remarked above).
Thus the space of probability measures on $B_1$ is itself metrizable and compact (using weak convergence of measures) 
and there is a sequence $r_n \to 0$ along which the limit 
\[
\Theta(df)  := \lim_{n \to \infty} 
Q_{r_n}(\{\mu: s_\mu \in df\})
\]
exists. Since $Q_{r_n}$ is supported on $\tilde{C}_\theta$ which is weak-$*$ closed, the limit measure $\Theta$ is supported on $\tilde{C}_{\theta}$. 
We now aim to pass to the limit $r_n \downarrow 0$ in (\ref{extreme10}) to produce
\begin{equation} \label{choquet}
\int_{\M_0} F(\mu) Q_t(d \mu) = 
\int_{\tilde{C}_{\theta}} \int_{\M_0} F(\nu) Q^f_{t}(d \nu) \Theta(df)
\end{equation}
This will be true provided both (i) $f \mapsto \int_{\M_0} F(\nu) Q^f_{t}(d \nu)$ is a bounded continuous function and 
(ii) $
\int_{\M_0} F(\nu) Q^f_{t-r}(d \nu) \mapsto \int_{\M_0} F(\nu) Q^f_{t}(d \nu)  \mbox{ uniformly over $B_1$
as $r \downarrow 0$.}
$
We will check (i) and (ii) for a law determining set of functions $F$, which implies that 
\begin{equation} \label{choquet2}
Q_t(d \nu) = \int_{\tilde{C}_{\theta}} Q^f_{t}(d \nu) \Theta(df)
\end{equation}
showing that any entrance law is a mixture as we hoped. 
We choose, for our choice of function $F$ in (\ref{choquet}), a Laplace functional defined, when $\nu = \sum_i \delta_{x_i} $, by  
$F_{\phi} (\nu) = \prod_i (1- \phi(x_i))$, for $\phi: \R \to [0,1)$ continuous and compactly supported. Then   
$\int F(\nu) Q^f_{t}(d \nu)$ will be a Fredholm Pfaffian, 
indeed 
\[
\int_{\M_0} F_{\phi}(\mu)  Q^f_t(d \mu) = 1 + \sum_{k=1}^{\infty} \frac{(-1)^k}{k!} \int_{\R^k} \prod_{i=1}^k  
\phi(x_i)
 \pf( \K^f_t(x_i,x_j): i,j \leq k) dx_1 \ldots dx_k.
\]
Each term in the series is a continuous function of $f \in B_1$ (see the explicit formula for $K_t(x,y)$ in (\ref{scalarkernel}))
and the infinite series can be controlled using the Hadamard bound on determinants:
\[
| \pf( \K_t(x_i,x_j): i,j \leq k) | = | \det( \K_t(x_i,x_j): i,j \leq k) |^{1/2} \leq \|\K_t\|_{\infty}^k (2k)^{k/2}
\]
and that $K^f_t$ and its derivatives are all bounded functions at $t>0$. 
Similar estimates, using the boundedness of time derivatives at $t>0$, establish the convergence in condition (ii) above. 

The representation (\ref{choquet2}) is not quite a Choquet representation for elements of a convex set, but 
it remains to identify the extremal elements of $(Q^f: f \in \tilde{C}_{\theta})$. We first explain the intuition of why, when $\theta \in [0,1)$ and $k \geq 2$, the entrance law $Q^{f_k}$ corresponding to $f_k (x,y) = (-\theta)^{k \, \I(x<0<y)}$ is not extremal. This should correspond to the limit of processes with $k$ particles at time $0$ with positions converging to $0$. These $k$ initial particles will very quickly react and, when the dust has settled, there will be either $0$ particles remaining with probability $p_k$, or $1$ particles remaining
with probability $1-p_k$. An easy calculation gives 
\[
p_k = (\theta+(-\theta)^k)/(1+\theta).
\]
Thus we should expect 
$ Q^{f_k} = p_k  Q^{f_1} + (1-p_k) Q^{f_0}$. The same happens for any $f \in \tilde{C}_{\theta}$ with a weight $w(a) \geq 2$ at an isolated point $a$. There turn out to be no other examples of non extremal entrance laws. Moreover, as we now check, all this can be easily verified using the
Pfaffian structure.

Suppose first that $\theta=1$. Take $\hat{f} = f \otimes f, \hat{g} = g \otimes g, \hat{h} = h \otimes h \in C_1$ and suppose 
\[
Q^{\hat{f}} = p Q^{\hat{g}}  + (1-p) Q^{\hat{h}} \quad \mbox{for some $p \in (0,1)$.}
\]
 The duality formula (\ref{duality}) shows that $\pf(K_t^{\hat{f}}) = p \pf(K^{\hat{g}}_t) + (1-p) \pf(K_t^{\hat{h}})$. 
 Letting $t \downarrow 0$, and using the fact that $\pf( f(x_i)f(x_j): i, j \leq 2n) = f(x_1) f(x_2) \ldots f(x_{2n})$, 
implies that 
\begin{equation} \label{equality}
\prod_{i=1}^{2n} f(x_i) = p \, \prod_{i=1}^{2n} g(x_i) + (1-p) \prod_{i=1}^{2n} h(x_i) \quad \mbox{for almost all $x_1<x_2< \ldots < x_{2n}$.}
\end{equation}
If $f,g,h$ were continuous functions we could take limits to conclude for all $x \in \R$ that 
\begin{equation} \label{equality2}
f(x)^n = p \, g(x)^{2n} + (1-p) h(x)^{2n} \quad \mbox{for all $n \geq 1$.}
\end{equation}
For only measurable $f,g,h$ we can apply Lusin's Theorem to conclude that (\ref{equality2}) still holds for almost all $x \in \R$. 
The only solution to this set of equations is $g(x) = \pm f(x)$ and $h(x) = \pm f(x)$ for almost all $x \in \R$, which implies that 
$\hat{f}=\hat{g}=\hat{h}$ establishing the extremality.

Now suppose $\theta \in [0,1)$. The entrance laws corresponding to $f \in \tilde{C}_{\theta}$ which have weights $w \equiv 1$ are extreme. They can be written as $f_S(x,y) = \I(S \cap (x,y) = \emptyset)$ where $S \subseteq \R$ is a closed set. Suppose 
$Q^{f_S}  = p Q^{f_{S_1}} + (1-p) Q^{f_{S_2}}$ for some $p \in (0,1)$ and closed sets $S,S_1,S_2$. As above, taking $t \downarrow 0$ in 
 the duality formula (\ref{duality}) when $n=1$, we reach
 \[
\I(S \cap (x,y) = \emptyset) = p \I(S_1 \cap (x,y) = \emptyset) + (1-p) \I(S_2 \cap (x,y) = \emptyset) \quad \mbox{for almost all $x,y$.}
 \]
Fix $ z \in S$. Then $\I(S \cap (x,y) = \emptyset) =0$ for almost all $x<z<y$ which implies
 $\I(S_1 \cap (x,y) = \emptyset) =\I(S_1 \cap (x,y) = \emptyset)=0$ for almost all $x<z<y$, and thus that $z \in S_1 \cap S_2$.
 Conversely,  if $z \in S^c$ and $d(z,S) > \epsilon$ then $\I(S \cap (x,y) = \emptyset) =1$ for almost all $z-\epsilon <x<y< z+\epsilon$ which implies
 $\I(S_1 \cap (x,y) = \emptyset) =\I(S_1 \cap (x,y) = \emptyset)=1$ for almost all  $z-\epsilon <x<y< z+\epsilon$ , and thus that $z \in S^c_1 \cap S^c_2$. We conclude that $S=S_1=S_2$ as desired. 

Finally we check that the entrance laws corresponding to $f \in \tilde{C}_{\theta}$ which have at least one isolated point $a$ with weight 
$w (a) \geq 2$ are not extreme. As suggested above, we claim that if 
$f_k(x,y) = \I (S_c \cap (x,y)  = \emptyset) (-\theta)^{\sum_{z \in S_i \cap(x,y)} w(z)}$ where 
$w(a) =k \geq 2$ then $Q^{f_k} = p_k Q^{f_1} + (1-p_k) Q^{f_0}$. It is enough to check that the duality formulae (\ref{duality}) 
add up correctly at all $t>0$, since these formulae determine the laws of a simple point process at a fixed time $t>0$. We use that the expectations 
$u^{(2n)}_t(x_1,\ldots,x_{2n}) = \E[ (-\theta)^{X_t(x_1,x_2) + X_t(x_3,x_4) + \ldots +X_t(x_{2n-1},x_{2n}) } ]$ 
are the unique bounded solutions to the linear p.d.e.'s: for $n \geq 1$
\begin{eqnarray*}
\partial_t u^{(2n)}_t(x) & = & \frac12 \Delta u^{(2n)}_t(x)  \quad \mbox{when $x_1<x_2< \ldots < x_{2n}$},\\
u^{(2n)}(x) & = &  u^{(2n-2)}_t(x\setminus \{x_k,x_{k+1}\}) \quad \mbox{when $x_1< \ldots x_k = x_{k+1} < \ldots <x_{2n}$},
\end{eqnarray*}
supplemented with the appropriate initial condition. The equation and the boundary conditions follow from
those for the $n=2$ case (\ref{Kpde}) and the Pfaffian formula  (\ref{duality}) for $u_t^{(2n)}$.
Thus we need only check that the initial conditions for these p.d.e.'s add up, that is 
\[
f_k(x_1,x_2) \ldots f_k(x_{2n-1},x_{2n}) = p_k f_0(x_1,x_2) \ldots f_k(x_{2n-1},x_{2n}) + (1-p_k) f_1(x_1,x_2) \ldots f_k(x_{2n-1},x_{2n}).
\]
However this just reduces to the simple identity $(-\theta)^k = p_k + (1-p_k) (-\theta)$. 
\qed 
 
 \textbf{Proof of Lemma \ref{closure}.} We argue separately for the cases $\theta =1$ and $\theta \in [0,1)$.  We will show  
(i) the closure of $S_\theta$ must contain $\tilde{C}_{\theta}$ and (ii) that $\tilde{C}_{\theta}$ are weak-$*$ closed. 

\textbf{Case $\theta =1$.}  In this case, for $\mu \in \M_0$, the spin function factorises $s_{\mu}(x,y) = \hat{s}_{\mu}(x) \hat{s}_{\mu}(y)$, for Lebesgue almost all $(x,y)$, where
\[
\hat{s}_{\mu}(x) = \left\{ 
\begin{array}{ll} 
(-1)^{\mu(0,x)} & \mbox{ for $x \geq 0$,} \\
(-1)^{\mu(x,0)} & \mbox{ for $x <0$.}
\end{array} \right.
\]
Indeed the factorization is an equality provided $x$ and $y$ avoid the support of $\mu$. 
Fix measurable $f:\R \to [-1,1]$.
We will construct below a sequence $(\mu_n: n \geq 1)$ of finite simple point measures so that 
$\hat{s}_{\mu_n} \to f $ using weak-$*$ convergence in $L^{\infty}(\R)$.
We write $f \otimes g$ for the function defined by $f \otimes g(x,y) := f(x) g(y)$.
Then the factorisation implies that 
$s_{\mu_n} \to f \otimes f$ using weak-$*$ 
convergence in $L^{\infty}(V_2)$ or in $L^{\infty}(\R^2)$.
Thus the weak-$*$ closure of $S_\theta$ contains $\tilde{C}_1$.

Define an approximation $f_n$ to $f$, for $n \geq 1$, by  
$f_n (x) = 1$ for $x \not \in [-n,n)$ and 
\[
f_n(x) = \left\{ 
\begin{array}{ll} 
1 & \mbox{ for $x \in [k/n, a_{k,n})$,} \\
-1 & \mbox{ for $x \in [a_{k,n}, (k+1)/n)$,}
\end{array} \right.
\quad \mbox{for $k = -n^2, \ldots, n^2-1$}
\]
where $a_{k,n}$ are chosen so that $\int^{(k+1)/n}_{k/n} f(x) dx = \int^{(k+1)/n}_{k/n} f_n(x)dx$.
Note $f_n = \hat{s}_{\mu_n}$ (almost everywhere) for a finite measure $\mu_n \in \M_0$. 
Fix $\phi \in L^1(\R)$ and $\epsilon >0$. Choose $\tilde{\phi}$ smooth 
compactly supported so that $\|\phi - \tilde{\phi}\|_{L^1} \leq \epsilon$. Then, if $\tilde{\phi}$ is supported in 
$[-L,L]$, 
\begin{eqnarray*}
|(f_n - f, \phi)| & \leq  &  |(f_n - f, \tilde{\phi})| + |(f_n - f, \tilde{\phi} - \phi)| \\
& \leq  &  |(f_n - f, \tilde{\phi})| + 2 \epsilon \\
& = & \sum_k |\int^{(k+1)/n}_{k/n} (f_n(x) - f(x))(\tilde{\phi}(x) - \tilde{\phi}(k/n)) dx| 
+ 2 \epsilon \\
& \leq & \frac{2(L+1)n}{n^2} \|\tilde{\phi}'\|_{\infty} + 2 \epsilon.
\end{eqnarray*}
This establishes the desired weak-$*$ convergence $\hat{s}_{\mu_n} = f_n  \to f$. 

To show that $\tilde{C}_1$ is closed we suppose $f_n \otimes f_n \to \psi$ weak-$*$ in $L^{\infty}(V_2)$ as $n \to \infty$.
By weak-$*$ compactness of the unit ball, we may choose a subsequence $n_k$
where $f_{n_k} \to f_{\infty}$ weak-$*$ in $L^{\infty}(\R)$ as $k \to \infty$. Then $f_{n_k} \otimes f_{n_k} \to f_{\infty} \otimes f_{\infty}$ 
weak-$*$ in $L^{\infty}(\R^2)$, and this identifies the limit point $\psi$ in the product form $\tilde{\psi} = 
 f_{\infty} \otimes f_{\infty}$ as desired.

\textbf{Case $\theta \in [0,1)$.} 
Let $\M_s$ be the set of $s$-finite counting measures, that is $\mu \in \M_s$ if it is a countable sum 
of finite point measures on $\R$.  We will below identify the weak-$*$ closure of $S_{\theta}$ 
as
\begin{equation}
\tilde{C}_{\theta} = \left\{ s_{\mu}(x,y): \mu \in \M_s \right\} \label{M_s}
\end{equation}
where $s_{\mu}$ is still defined as in (\ref{thetaduality}), understanding that $(-\theta)^{\infty}=0$. 
The representation (\ref{M_s}) of a limit points is not unique; but the function
$s_{\mu}$ is uniquely identified via the closed support $S=S_c \cup S_i$ of $\mu$ and the mass
$\mu(\{a\})$ of any isolated point $a \in S_i$. 
Indeed, noting that the isolated points $S_i$ are locally finite,
for  $(x,y) \in V_2$,
\[
 s_{\mu}(x,y) =  (-\theta)^{\mu(x,y)} = \left\{ \begin{array}{ll}
 0 & \mbox{if $(x,y) \cap S_c \neq \emptyset$,} \\
 (-\theta)^{\sum_{a \in S_i \cap(x,y)} \mu(\{a\})} & \mbox{if $(x,y) \cap S_c \neq \emptyset$.}
 \end{array} \right.
\]
Thus the set agrees with the formula for $\tilde{C}_{\theta}$ given in Lemma \ref{closure} when $w(a) = \mu(\{a\})$. Moreover, the new formula, indexed by $S$ and $w$, gives distinct functions $s_{\mu}(x,y)$ in $L^2(V_2)$, 
so that this identification is bijective. 

To establish (\ref{M_s}) we may write any $\mu \in \M_s$ as $\mu =\sum_{k=1}^{\infty} \delta_{x_k}$, where
the positions $(x_k)$ are not necessarily disjoint. We can find
$\sum_{k=1}^{\infty} \delta_{x_{k,n}}$ where the sequences $(x_{k,n}: k \geq 1)$ do have disjoint elements and where, for each $k \geq 1$ we have $\lim_{n \to \infty} x_{k,n} = x_k$.  We set 
$\mu_n = \sum_{k=1}^n \delta_{x_{k,n}}$. Then $s_{\mu_n} \to s_{\mu}$ weak-$*$ showing that $\tilde{C}_{\theta}$ is contained in the weak-$*$ closure. 
To see that the set $\tilde{C}_{\theta} $ is itself weak-$*$ closed we suppose that
$\mu_n \in \M_s$ and $s_{\mu_n} \to \psi$ weak-$*$. For each $n$ the measure $\mu_n$ can be written
as $\mu_n = \sum_{k=1}^{\infty} \mu_{k,n}$ where
each $\mu_{k,n}$ is either zero or a single point mass supported in $[-k,k]$. Then by diagonalisation we can find a sub-sequence $n'$ along which the measures $\mu_{k,n'}$ are weakly convergent for all $k \geq 1$. Set
$\mu_{k,\infty} = \lim_{n' \to \infty} \mu_{k,n'}$ and $\mu_{\infty} = \sum_k \mu_{k,\infty} \in \M_s$. Then $s_{\mu_{n'}} \to s_{\mu_{\infty}}$ weak-$*$, confirming that $\tilde{C}_{\theta}$ is closed. \qed

\bibliographystyle{abbrv}
\bibliography{Elaw_bib}

\end{document}